\documentclass[12pt,reqno]{article}

\usepackage[usenames]{color}
\usepackage[colorlinks=true,
linkcolor=webgreen, filecolor=webbrown,
citecolor=webgreen]{hyperref}

\definecolor{webgreen}{rgb}{0,.5,0}
\definecolor{webbrown}{rgb}{.6,0,0}

\usepackage{amsmath}
\usepackage{amssymb}
\usepackage{mathtools}
\usepackage{graphicx}
\usepackage{amscd}
\usepackage{lscape}
\usepackage{tikz}
\usepackage{tikz-cd}
\usepackage{pgfplots}

\usetikzlibrary{matrix}
\usetikzlibrary{fit,shapes}
\usetikzlibrary{positioning, calc}
\tikzset{circle node/.style = {circle,inner sep=1pt,draw, fill=white},
        X node/.style = {fill=white, inner sep=1pt},
        dot node/.style = {circle, draw, inner sep=5pt}
        }
\usepackage{tkz-fct}

\usepackage{amsthm}
\newtheorem{theorem}{Theorem}

\newtheorem{proposition}[theorem]{Proposition}
\newtheorem{corollary}[theorem]{Corollary}

\theoremstyle{definition}

\newtheorem{example}[theorem]{Example}

\usepackage{float}

\usepackage{graphics,amsmath}
\usepackage{amsfonts}
\usepackage{latexsym}
\usepackage{epsf}

\newcommand{\seqnum}[1]{\href{http://oeis.org/#1}{\underline{#1}}}

\begin{document}

\begin{center}
\vskip 1cm{\LARGE\bf Square roots in the Appell group and Sprugnoli arrays} \vskip 1cm \large
Paul Barry\\
School of Science\\
South East Technological University\\
Ireland\\
\href{mailto:pbarry@wit.ie}{\tt pbarry@wit.ie}
\end{center}
\vskip .2 in

\begin{abstract} We introduce a special mapping from pairs of power series to the group of Sprugnoli matrices. This mapping has the property when the second argument is an even power series, then the square of the resulting Sprugnoli array is an aerated element of the Appell subgroup of the Riordan group. This allows us to explore the square roots of elements in the aerated Appell subgroup. As the identity is an element of this subgroup, we are led to explore related involutions in the Sprugnoli group.
\end{abstract}

\section{Introduction}
In \cite{SQ}, Merlini investigates the square root of a Bell matrix, which is a matrix of the form $(g(x), xg(x))$ in the Riordan group \cite{book1, book2, SGWW}. In this note, we shall investigate the square root of certain elements of the aerated Appell subgroup of the Riordan group. The special feature of these square roots is that they are Sprugnoli arrays, elements of the Sprugnoli group \cite{Spru}.
For this, we define 
$$\mathcal{F}_r =\{h(x) \in \mathbb{C}[[x]] \,|\, h(x)=\sum_{i=r}^{\infty} h_i x^i\}.$$
For $(g,f) \in \mathcal{F}_0 \times \mathcal{F}_0$, we define $M(g,f)$ to be the matrix with generating function given by 
$$\frac{g(x)}{(1-x)(1+yx)}+\frac{xf(x)}{(1-x^2)(1-yx)}.$$ We shall show that $M(g,f)$ is a Sprugnoli array, and moreover, when $f(x)$ is an even power series, that $M(g,f)^2$ is a member of the Appell subgroup of the Riordan group. In fact, it is an aerated element of the Appell subgroup. For certain pairs $(g,f)$, we have $M(g,f)^2=I$, the identity matrix. Thus in such a case, $M(g,f)$ represents an involution in the Sprugnoli group.

We shall prove these assertions after we have laid some groundwork by explaining the necessary facts about the Riordan group, its subgroup of Appell arrays, and the necessary facts about the Sprugnoli group. The structure of this note is as follows.
\begin{enumerate}
\item This introduction.
\item Examples.
\item The Riordan group $R$ and its subgroup of Appell arrays.
\item The Sprugnoli group $S$.
\item The mapping $\mathcal{F}_0 \times \mathcal{F}_0 \longrightarrow S$.
\item Squaring the Sprugnoli array $M(g,f)$. 
\item Involutions in the Sprugnoli group.
\item A Riordan group interpretation of the special mapping.
\item Conclusions.
\end{enumerate}
Note that where integer sequences occur in this note that have entries in the On-Line Encyclopedia of Integer Sequences (OEIS) \cite{SL1, SL2}, we refer to them by their OEIS reference. 
\section{Examples}
In this section, we use truncations of the infinite matrices involved to illustrate the notions that will be explored in this note. In all cases, the matrices on the left are Sprugnoli arrays, while those on the right are aerated elements of the Appell subgroup of the Riordan group.
\begin{example}
\[
\begin{pmatrix}
1 & 0 & 0 & 0 & 0 & 0 & 0 \\
2 & -1 & 0 & 0 & 0 & 0 & 0 \\
\frac12 & 0 & 1 & 0 & 0 & 0 & 0 \\
3 & -\frac12 & 2 & -1 & 0 & 0 & 0 \\
-\frac{9}{8} & 2 & \frac12 & 0 & 1 & 0 & 0 \\
\frac{15}{4} & \frac{9}{8} & 3 & -\frac12 & 2 & -1 & 0 \\
-\frac{111}{16} & 6 & -\frac{9}{8} & 2 & \frac12 & 0 & 1
\end{pmatrix}^2=\begin{pmatrix}
1 & 0 & 0 & 0 & 0 & 0 & 0 \\
0 & 1 & 0 & 0 & 0 & 0 & 0 \\
1 & 0 & 1 & 0 & 0 & 0 & 0 \\
0 & 1 & 0 & 1 & 0 & 0 & 0 \\
2 & 0 & 1 & 0 & 1 & 0 & 0 \\
0 & 2 & 0 & 1 & 0 & 1 & 0 \\
3 & 0 & 2 & 0 & 1 & 0 & 1
\end{pmatrix}
\]
The array on the right is a representation of the Appell array $\left(\frac{1}{1-x^2-x^4}, x\right)$, defined by the aerated Fibonacci numbers.
\end{example}
\begin{example} In this example, we let $r \in \mathbb{C}$. We have 
\begin{scriptsize}
\[
\begin{pmatrix}
1 & 0 & 0 & 0 & 0 & 0 & 0 \\
-r + 3 + i & -1 & 0 & 0 & 0 & 0 & 0 \\
-2(r-1) + i(2-r) & r-1-i & 1 & 0 & 0 & 0 & 0 \\
-2r + 3 + 2i(1-r) & 2(r-1) + i(r-2) & -r + 3 + i & -1 & 0 & 0 & 0 \\
-2r + 3 + 2i(1-r) & 2r - 1 + 2i(r-1) & -2(r-1) + i(2-r) & r-1-i & 1 & 0 & 0 \\
-3r + 5 + i(3-2r) & 2r - 3 + 2i(r-1) & -2r + 3 + 2i(1-r) & 2(r-1) + i(r-2) & -r + 3 + i & -1 & 0 \\
-4(r-1) + i(4-3r) & 3(r-1) + i(2r-3) & -2r + 3 + 2i(1-r) & 2r - 1 + 2i(r-1) & -2(r-1) + i(2-r) & r-1-i & 1
\end{pmatrix}^2
\]\end{scriptsize}
\[=
\begin{pmatrix}
1 & 0 & 0 & 0 & 0 & 0 & 0 \\
0 & 1 & 0 & 0 & 0 & 0 & 0 \\
2 - r^{2} & 0 & 1 & 0 & 0 & 0 & 0 \\
0 & 2 - r^{2} & 0 & 1 & 0 & 0 & 0 \\
4 - r^{2} & 0 & 2 - r^{2} & 0 & 1 & 0 & 0 \\
0 & 4 - r^{2} & 0 & 2 - r^{2} & 0 & 1 & 0 \\
5 - 2r^{2} & 0 & 4 - r^{2} & 0 & 2 - r^{2} & 0 & 1
\end{pmatrix}.
\]
This is the aerated Appell matrix $\left(\frac{1+x^2+x^4}{(1-x^2)(1-x^4)}-\frac{r^2x^2(1+x^2)}{(1-x^4)^2},x\right)$. When $r=\sqrt{3}$, we obtain the matrix
\[
\begin{pmatrix}
1 & 0 & 0 & 0 & 0 & 0 & 0 \\
0 & 1 & 0 & 0 & 0 & 0 & 0 \\
-1 & 0 & 1 & 0 & 0 & 0 & 0 \\
0 & -1 & 0 & 1 & 0 & 0 & 0 \\
1 & 0 & -1 & 0 & 1 & 0 & 0 \\
0 & 1 & 0 & -1 & 0 & 1 & 0 \\
-1 & 0 & 1 & 0 & -1 & 0 & 1
\end{pmatrix},
\]
which is the Appell matrix $\left(\frac{1}{1+x^2}, x\right)$.

\end{example}
\begin{example} 
We have 
\[
\begin{pmatrix}
1 & 0 & 0 & 0 & 0 & 0 & 0 \\
2 & -1 & 0 & 0 & 0 & 0 & 0 \\
2 & 0 & 1 & 0 & 0 & 0 & 0 \\
4 & -2 & 2 & -1 & 0 & 0 & 0 \\
3 & 0 & 2 & 0 & 1 & 0 & 0 \\
6 & -3 & 4 & -2 & 2 & -1 & 0 \\
4 & 0 & 3 & 0 & 2 & 0 & 1
\end{pmatrix}^2=\begin{pmatrix}
1 & 0 & 0 & 0 & 0 & 0 & 0 \\
0 & 1 & 0 & 0 & 0 & 0 & 0 \\
4 & 0 & 1 & 0 & 0 & 0 & 0 \\
0 & 4 & 0 & 1 & 0 & 0 & 0 \\
10 & 0 & 4 & 0 & 1 & 0 & 0 \\
0 & 10 & 0 & 4 & 0 & 1 & 0 \\
20 & 0 & 10 & 0 & 4 & 0 & 1
\end{pmatrix}
\]
This is the aerated Appell array $\left(\frac{1}{(1-x^2)^4},x\right)$.
\end{example}

\section{The Riordan group}
We recall that $\mathcal{F}_0 = \{ \sum_{n=0}^{\infty}a_n x^n\,|\, a_0 \ne 0\}$ and that $\mathcal{F}_1 = \{ \sum_{n=0}^{\infty}a_n x^n\,|\, a_0 =0, a_1 \ne 0\}$ and in general
$$\mathcal{F}_r = \{\sum_{n=r}^{\infty}a_n x^n\,|\, a_r \ne 0\}.$$ Elements of $\mathcal{F}_0$ are multiplicatively invertible, and elements of $\mathcal{F}_1$ are compositionally invertible, given suitable ground rings $R$ for $a_n \in R$. In the sequel we shall always take this ring to be the field of complex numbers $\mathbb{C}$, and we shall reserve the symbol $R$ for the Riordan group thus defined.

The Riordan group is then the group of pairs $(g, f) \in \mathcal{F}_0 \times \mathcal{F}_1$ with the following product rule
\begin{equation}\label{product}(g(x), f(x))\cdot (u(x), v(x))=  (g(x)u(f(x)), v(f(x))\end{equation} and inverse
$$(g, f)^{-1}= \left(\frac{1}{g(\bar{f})}, \bar{f}(x)\right),$$ where $\bar{f}$ is the compositional inverse of $f \in \mathcal{F}_0$ (that is, $\bar{f}$ is the solution $u$ of the equation $f(u)=x$ for which $u(0)=0$). The identity of this group is $(1,x)$. To each element of this group we can associate in a unique way a lower-triangular matrix $(t_{n,k})$ with entries in the ground ring (in this case $\mathbb{C}$) by means of
$$t_{n,k}=[x^n] g(x)f(x)^k,$$ where $[x^n]$ is the functional on $\mathbb{C}[[x]]$ that extracts the coefficient of $x^n$. Under this correspondence, the product (\ref{product}) corresponds to ordinary matrix multiplication. The columns of $(g,f)$ have their generating functions given by the geometric series of power series
$$g, gf, gf^2, gf^3, gf^4, \ldots.$$ The bi-variate generating function of the Riordan array $(g(x), f(x))$ is given by $\frac{g(x)}{1-yf(x)}$.
\begin{example} The Riordan array $(g(x),f(x))=\left(\frac{1}{1-x}, \frac{x}{(1-x)^2}\right)$ \seqnum{A085478} begins
$$\left(\begin{array}{ccccccc}
1 & 0 & 0 & 0 & 0 & 0 & 0 \\
1 & 1 & 0 & 0 & 0 & 0 & 0 \\
1 & 3 & 1 & 0 & 0 & 0 & 0 \\
1 & 6 & 5 & 1 & 0 & 0 & 0 \\
1 & 10 & 15 & 7 & 1 & 0 & 0 \\
1 & 15 & 35 & 28 & 9 & 1 & 0 \\
1 & 21 & 70 & 84 & 45 & 11 & 1\\
\end{array}\right).$$
We have $t_{n,k}=\binom{n+k}{2k}$.
The generating function of this array is given by
$$\frac{\frac{1}{1-x}}{1-y\frac{x}{(1-x)^2}}=\frac{1-x}{1-x(y+2)+x^2}.$$
When $y=1$, we obtain the generating function $\frac{1-x}{1-3x+x^2}$ of the row sums of this matrix. Thus the row sums begin
$$1, 2, 5, 13, 34, 89, 233, 610, 1597, 4181, 10946,\ldots,$$ or $F_{2n+1}$ \seqnum{A122367}.
\end{example}  The product law follows from the following result, called the ``fundamental theorem of Riordan arrays'', which details how a Riordan array operates on a power series. We have
$$(g(x), f(x))\cdot h(x)= g(x)h(f(x)).$$ This is sometimes called a weighted composition. 

The Appell subgroup of the Riordan group is the set of elements $(g(x),x)$ with $g(x) \in \mathcal{F}_0$. This subgroup itself has a subgroup, the group of ``aerated'' Appell arrays $(g(x^2), x)$ where $g(x) \in \mathcal{F}_0$. We note that the unit of the Riordan group is an element of these subgroups. 
\begin{example} The Riordan array $M=\left(\frac{1}{1-x}, -\frac{x}{1-x}\right)$ has general element $(-1)^k \binom{n}{k}$. We have $M^2=I$. We say that $M$ is an involution (element of order $2$) in the Riordan group.

\[
\begin{pmatrix}
1 & 0 & 0 & 0 & 0 & 0 & 0 \\
1 & -1 & 0 & 0 & 0 & 0 & 0 \\
1 & -2 & 1 & 0 & 0 & 0 & 0 \\
1 & -3 & 3 & -1 & 0 & 0 & 0 \\
1 & -4 & 6 & -4 & 1 & 0 & 0 \\
1 & -5 & 10 & -10 & 5 & -1 & 0 \\
1 & -6 & 15 & -20 & 15 & -6 & 1
\end{pmatrix}^2=\begin{pmatrix}
1 & 0 & 0 & 0 & 0 & 0 & 0 \\
0 & 1 & 0 & 0 & 0 & 0 & 0 \\
0 & 0 & 1 & 0 & 0 & 0 & 0 \\
0 & 0 & 0 & 1 & 0 & 0 & 0 \\
0 & 0 & 0 & 0 & 1 & 0 & 0 \\
0 & 0 & 0 & 0 & 0 & 1 & 0 \\
0 & 0 & 0 & 0 & 0 & 0 & 1
\end{pmatrix}.
\]
\end{example}

\begin{example} The matrix that begins 
\[
\begin{pmatrix}
1 & 0 & 0 & 0 & 0 & 0 & 0 \\
1 & 0 & 0 & 0 & 0 & 0 & 0 \\
1 & 1 & 0 & 0 & 0 & 0 & 0 \\
1 & 2 & 0 & 0 & 0 & 0 & 0 \\
1 & 3 & 1 & 0 & 0 & 0 & 0 \\
1 & 4 & 3 & 0 & 0 & 0 & 0 \\
1 & 5 & 6 & 1 & 0 & 0 & 0
\end{pmatrix}
\]
is an example of a ``stretched'' Riordan array, which can be represented by the pair $\left(\frac{1}{1-x}, \frac{x^2}{1-x}\right)$. Stretched Riordan arrays are defined by $(g,f) \in \mathcal{F}_0 \times \mathcal{F}_2$, where again the general $(n,k)$ element is given by 
$t_{n,k}=[x^n] g(x)f(x)^k$. 
\end{example}

\section{The Sprugnoli Group}
The \emph{Sprugnoli group} $S$ is the set with elements $(g, f_1, f_2)$ with $g(x) \in \mathcal{F}_0$, $f_1(x) \in \mathcal{F}_1$, and $f_2(x) \in \mathcal{F}_1$ and
$f_2 \in xR[[x^2]]$ (thus $f_2$ is an odd power series). The element $(g, f_1, f_2)$ of this group has the matrix representation
$$t_{n,k} = [x^n] g(x)f_1(x)^{k \bmod 2} (xf_2(x))^{\lfloor \frac{k}{2} \rfloor}.$$ (We take the coefficient ring of power series to be the field of complex numbers $\mathbb{C}$).

We have
$$
t_{n,k} =
\begin{cases}
  [x^n]g(x)(xf_2(x))^m, & k=2m, \\
  [x^n]g(x)f_1(x)(xf_2(x))^m,  & k=2m+1.
\end{cases}
$$
The columns of this matrix then have generating functions given by the following products of generating functions
\begin{align*}g(x), g(x)f_1(x),&\, g(x)(xf_2(x)), g(x)f_1(x)(xf_2(x)), g(x)(xf_2(x))^2,\\
& g(x)f_1(x)(xf_2(x))^2, g(x)(xf_2(x))^3, g(x)f_1(x)(xf_2(x))^3,\ldots\end{align*} We can represent this sequence by the following schema.
$$
\begin{array}{cccccccc}
g &g & g & g & g & g & g & g\\
1 &f_1 & 1 & f_1 & 1 & f_1 & 1 & f_1 \\
1 &1 & f_2 & f_2 & f_2^2 & f_2^2 & f_2^3 & f_2^3 \\
1 & 1 & x & x & x^2 & x^2 & x^3 & x^3 \\
- & - & - & - & - & - & - & -\\
0 & 1 & 2 & 3 & 4 & 5 & 6 & 7 \\
\end{array}
$$ Here, the final row gives the order of the product of the terms above it, showing that this associated matrix is lower-triangular. We can represent this array as the sum of two matrices as follows. The first matrix is the matrix whose columns are generated by the power series
$$g(x), 0, g(x)(xf_2(x)), 0, g(x)(xf_2(x))^2, \ldots,$$ which is a horizontal aeration of the stretched Riordan array $(g(x), xf_2(x))$, and the second matrix is the matrix whose columns are generated by the power series
$$0,g(x)f_1(x),0, g(x)f_1(x)(xf_2(x)), 0, g(x)f_1(x)(xf_2(x))^2,0,\ldots,$$ which is the horizontal aeration of the stretched Riordan array $(g(x)f_1(x), xf_2(x))$ (but note that this matrix starts with a zero row, due to the fact that $f_1 \in \mathcal{F}_1$).
\begin{example} We consider the Sprugnoli array $\left(\frac{1}{1-x}, \frac{x(1+x)}{1-x}, \frac{x}{1-x^2}\right)$. This array begins
$$\left(\begin{array}{ccccccccc}
1 & 0 & 0 & 0 & 0 & 0 & 0 & 0 & 0 \\
1 & 1 & 0 & 0 & 0 & 0 & 0 & 0 & 0 \\
1 & 3 & 1 & 0 & 0 & 0 & 0 & 0 & 0 \\
1 & 5 & 1 & 1 & 0 & 0 & 0 & 0 & 0 \\
1 & 7 & 2 & 3 & 1 & 0 & 0 & 0 & 0 \\
1 & 9 & 2 & 6 & 1 & 1 & 0 & 0 & 0 \\
1 & 11 & 3 & 10 & 3 & 3 & 1 & 0 & 0 \\
1 & 13 & 3 & 15 & 3 & 7 & 1 & 1 & 0 \\
1 & 15 & 4 & 21 & 6 & 13 & 4 & 3 & 1
\end{array}\right).$$ This is the sum of the matrices
$$\left(\begin{array}{ccccccccc}
1 & 0 & 0 & 0 & 0 & 0 & 0 & 0 & 0 \\
1 & 0 & 0 & 0 & 0 & 0 & 0 & 0 & 0 \\
1 & 0 & 1 & 0 & 0 & 0 & 0 & 0 & 0 \\
1 & 0 & 1 & 0 & 0 & 0 & 0 & 0 & 0 \\
1 & 0 & 2 & 0 & 1 & 0 & 0 & 0 & 0 \\
1 & 0 & 2 & 0 & 1 & 0 & 0 & 0 & 0 \\
1 & 0 & 3 & 0 & 3 & 0 & 1 & 0 & 0 \\
1 & 0 & 3 & 0 & 3 & 0 & 1 & 0 & 0 \\
1 & 0 & 4 & 0 & 6 & 0 & 4 & 0 & 1
\end{array}\right)+
\left(\begin{array}{ccccccccc}
0 & 0 & 0 & 0 & 0 & 0 & 0 & 0 & 0 \\
0 & 1 & 0 & 0 & 0 & 0 & 0 & 0 & 0 \\
0 & 3 & 0 & 0 & 0 & 0 & 0 & 0 & 0 \\
0 & 5 & 0 & 1 & 0 & 0 & 0 & 0 & 0 \\
0 & 7 & 0 & 3 & 0 & 0 & 0 & 0 & 0 \\
0 & 9 & 0 & 6 & 0 & 1 & 0 & 0 & 0 \\
0 & 11 & 0 & 10 & 0 & 3 & 0 & 0 & 0 \\
0 & 13 & 0 & 15 & 0 & 7 & 0 & 1 & 0 \\
0 & 15 & 0 & 21 & 0 & 13 & 0 & 3 & 0
\end{array}\right).$$
The first is a horizontal aeration of the stretched Riordan array $\left(\frac{1}{1-x},\frac{x^2}{1-x^2}\right)$, while the second one
is a horizontal aeration of the stretched Riordan array $\left(\frac{(1+x)}{(1-x)^2},\frac{x^2}{1-x^2}\right)$ with an extra initial column of zeros.
\end{example}
A special element of this set is given by $(1,x,x)$. Its columns are generated by the sequence $1,x,x^2,x^3,\ldots$ and so its matrix representation is  given by the usual identity matrix.

The bivariate generating function of the Sprugnoli array $(g, f_1, f_2)$ is given by
$$\frac{g(x)}{1-y^2xf_2(x)}+\frac{yg(x)f_1(x)}{1-y^2xf_2(x)}=\frac{g(x)(1+yf_1(x))}{1-y^2xf_2(x)}.$$
Note in particular that if the Sprugnoli array is of the form $(g(x), f_1(x), x)$ then its bivariate generating function is of the form 
$\frac{g(x)(1+yf_1(x)}{1-y^2x^2}$. 

The row sums of the Sprugnoli array $(g, f_1, f_2)$ have generating function $\frac{g(x)(1+f_1(x))}{1-xf_2(x)}$.
This results by setting $y=1$ in the generating function of the array.

The bivariate generating function of the Sprugnoli array $(g(x),x,x)$ is given by 
$$\frac{g(x)(1+yx)}{1-y^2x^2}=\frac{g(x)}{1-yx},$$ which is the generating function of the Appell array $(g(x),x)$. Thus the Sprugnoli array $(g(x),x,x)$ is also a Riordan array, namely the Appell array $(g(x),x)$.

The ``Fundamental theorem of Sprugnoli arrays'' details how a Sprugnoli array operates on a power series. We have
\[
(g(x), f_1(x), f_2(x))\cdot h(x)=g(x)h^e(xf_2(x))+g(x)f_1(x)h^o(xf_2(x)).
\]
Here, \begin{align*}h^e(x)&=\frac{h(\sqrt{x})+h(-\sqrt{x})}{2},\\
h^o(x)&=\frac{h(\sqrt{x})-h(-\sqrt{x})}{2\sqrt{x}}\end{align*}
are the even and odd bisections of $h(x)$.

We use this result to define the product of two Sprugnoli arrays.
$$(g,f_1,f_2)\cdot(u,v_1,v_2)=\left((g,f_1,f_2)\cdot u, \frac{(g,f_1,f_2)\cdot uv_1}{(g,f_1,f_2)\cdot u}, \frac{1}{x}\frac{(g,f_1,f_2)\cdot uxv_2}{(g,f_1,f_2)\cdot u}\right).$$

The inverse of a Sprugnoli array is then defined by 
$$(g,f_1,f_2)^{-1}=\left(1,\frac{x-f_1^e(\left(\overline{\sqrt{xf_2}}\right)^2)}{f_1^o(\left(\overline{\sqrt{xf_2}}\right)^2)},\frac{1}{x}\left(\overline{\sqrt{xf_2}}\right)^2\right)\cdot \left(\frac{1}{g},x,x\right).$$

\section{The mapping $\mathcal{F}_0 \times \mathcal{F}_0 \longrightarrow S$}
Given a pair of power series $(g,f) \in \mathcal{F}_0 \times \mathcal{F}_0$, we form the bivariate generating function 
$$B(x,y)=\frac{g(x)}{(1-x)(1+yx)}+\frac{xf(x)}{(1-x^2)(1-yx)},$$ and we designate by $M(g,f)$ the matrix with generating function $B(x,y)$. 
We have the following result.
\begin{proposition}
The matrix $M(g,f)$ is the Sprugnoli array 
$$(G, F_1, F_2)=\left(\frac{g(x)}{1-x}+\frac{xf(x)}{1-x^2}, \frac{x(xf(x)-(1+x)g(x))}{xf(x)+(1+x)g(x)},x\right).$$
\end{proposition}
\begin{proof}
We have 
\begin{align*}
B(x,y)&=\frac{g(x)}{(1-x)(1+yx)}+\frac{xf(x)}{(1-x^2)(1-yx)}\\
&=\frac{g(x)(1-yx)}{(1-x)(1-y^2x^2)}+\frac{xf(x)(1+yx)}{(1-x^2)(1-y^2x^2)}\\
&=\frac{1}{1-y^2x^2} \left(\frac{g(x)}{1-x}+\frac{xf(x)}{1-x^2}+y\left(\frac{xf(x)}{1-x^2}-\frac{xg(x)}{1-x}\right)\right)\\
&=\frac{1}{1-y^2x^2}\left(\frac{g(x)}{1-x}+\frac{xf(x)}{1-x^2}\right)\left(1+\frac{yx\left(\frac{xf(x)}{1-x^2}-\frac{g(x)}{1-x}\right)}{\frac{xf(x)}{1-x^2}+\frac{g(x)}{1-x}}\right)\\
&=\frac{1}{1-y^2x^2}\left(\frac{g(x)}{1-x}+\frac{xf(x)}{1-x^2}\right)\left(1+yx\frac{xf(x)-(1+x)g(x)}{xf(x)+(1+x)g(x)}\right).\end{align*}
\end{proof}

\begin{example}
When $g(x)=\frac{1}{1-x-x^2}$, and $f(x)=1$, we obtain the Sprugnoli array 
$$\left(\frac{1+2x-x^2-x^3}{(1-x-x^2)(1-x^2)}, \frac{-x(1+x^2+x^3)}{1+2x-x^2-x^3},x\right),$$ that begins
\[
\begin{pmatrix}
1 & 0 & 0 & 0 & 0 & 0 & 0 \\
3 & -1 & 0 & 0 & 0 & 0 & 0 \\
4 & -1 & 1 & 0 & 0 & 0 & 0 \\
8 & -4 & 3 & -1 & 0 & 0 & 0 \\
12 & -6 & 4 & -1 & 1 & 0 & 0 \\
21 & -12 & 8 & -4 & 3 & -1 & 0 \\
33 & -19 & 12 & -6 & 4 & -1 & 1
\end{pmatrix}
\]
\end{example}
\begin{example}
If $g(x)=\frac{1}{1-x-x^2}$, the generating function of the Fibonacci numbers $F_{n+1}$ \seqnum{A000045}, and $f(x)=1$,  then $G(x)=\frac{g(x)}{1-x}+\frac{x}{1-x^2}$ is the generating function of the sequence 
$$a_n=\sum_{k=0}^n F_{k+1} + \frac{1-(-1)^n}{2}.$$ This sequence begins 
$$1, 3, 4, 8, 12, 21, 33, 55, 88, 144, 232,\ldots.$$ (See \seqnum{A052952} and \seqnum{A074331}). 
\end{example}
\begin{example} $B(x,1)$ is the generating function of the row sums of $M(g,f)$. We have 
$$B(x,1)=\frac{g(x)}{1-x^2}+\frac{xf(x)}{(1-x)(1-x^2)}.$$ 
When $g(x)=\frac{1}{1-x-x^2}$ and $f(x)=1$, this is the generating function of the sequence 
$$\sum_{k=0}^{\lfloor \frac{n}{2} \rfloor}F_{n-2k+1} + \sum_{k=0}^n \frac{1-(-1)^k}{2}.$$ This sequence begins 
$$1, 2, 4, 6, 10, 15, 24, 37, 59, 93, 149,\ldots.$$ This is \seqnum{A167270}.
In general, $B(x,1)$ is the generating function of the sequence with general term 
$$\sum_{k=0}^{\lfloor \frac{n}{2} \rfloor} g_{n-2k} + \sum_{k=0}^{n-1} \sum_{j=0}^{\lfloor \frac{k}{2} \rfloor} f_{k-2j}.$$ 
\end{example}
\begin{example} The square of the above matrix is given by 
\[
\begin{pmatrix}
1 & 0 & 0 & 0 & 0 & 0 & 0 \\
3 & -1 & 0 & 0 & 0 & 0 & 0 \\
4 & -1 & 1 & 0 & 0 & 0 & 0 \\
8 & -4 & 3 & -1 & 0 & 0 & 0 \\
12 & -6 & 4 & -1 & 1 & 0 & 0 \\
21 & -12 & 8 & -4 & 3 & -1 & 0 \\
33 & -19 & 12 & -6 & 4 & -1 & 1
\end{pmatrix}^2=\begin{pmatrix}
1 & 0 & 0 & 0 & 0 & 0 & 0 \\
0 & 1 & 0 & 0 & 0 & 0 & 0 \\
5 & 0 & 1 & 0 & 0 & 0 & 0 \\
0 & 5 & 0 & 1 & 0 & 0 & 0 \\
14 & 0 & 5 & 0 & 1 & 0 & 0 \\
0 & 14 & 0 & 5 & 0 & 1 & 0 \\
36 & 0 & 14 & 0 & 5 & 0 & 1
\end{pmatrix}.
\]
This is the Appell array $$\left(\frac{1-3x^4+x^6}{(1-x-x^2)(1+x-x^2)(1-x^2)^2},x\right).$$
\end{example}

\section{Squaring the Sprugnoli array $M(g,f)$}
We have $M(g,f)=(G, F_1, F_2)$ with 
\begin{align*}
G(x)&=\frac{g(x)}{1-x}+\frac{xf(x)}{1-x^2}\\
F_1(x)&=\frac{x(xf(x)-(1+x)g(x))}{xf(x)+(1+x)g(x)}\\
F_2(x)&=x.\end{align*}
Then 
\begin{align*}M(g,f)^2&=(G, F_1, x)\cdot (G, F_1, x)\\
&=\left((G, F_1, x)\cdot G, \frac{(G, F_1, x)\cdot GF_1}{(G, F_1, x)\cdot G}, x\right).\end{align*}
We have the following result.
\begin{proposition} When $f(x) \in \mathcal{F}_0$ is an even power series, then  
$$M(g,f)^2=\left(\frac{g(x)g(-x)+x^2 f(x)^2}{(1-x^2)^2},x\right),$$ where the array on the right hand side is an aerated Appell array.
\end{proposition}
\begin{proof}
We begin by calculating $(G, F_1, x)\cdot G$. By the fundamental theorem of Sprugnoli arrays, we have 
$$(G, F_1, x)\cdot G = G(x) G^e(x^2)+G(x)F_1(x)G^o(x^2).$$
Now 
\begin{align*}
G^e(x^2)&=\left(\frac{g(x)}{1-x}+\frac{xf(x)}{1-x^2}+\frac{g(-x)}{1+x}+\frac{-xf(-x)}{1-x^2}\right)/2\\
&=\left(\frac{g(x)}{1-x}+\frac{xf(x)}{1-x^2}+\frac{g(-x)}{1+x}-\frac{xf(x)}{1-x^2}\right)/2\quad (f(-x)=f(x))\\
&=\left(\frac{g(x)}{1-x}+\frac{g(-x)}{1+x}\right)/2.\end{align*}
Also,
\begin{align*}
G^o(x^2)&=\frac{1}{2x}\left(\frac{g(x)}{1-x}+\frac{xf(x)}{1-x^2}-\frac{g(-x)}{1+x}+\frac{xf(x)}{1-x^2}\right)\\
&=\frac{1}{2x}\left(\frac{g(x)}{1-x}-\frac{g(-x)}{1+x}+\frac{2xf(x)}{1-x^2}\right). \end{align*}
Then 
\begin{align*}
(G, F_1, x)\cdot G &= G(x) G^e(x^2)+G(x)F_1(x)G^o(x^2)\\
&=\frac{G(x)}{2}\left(\frac{g(x)}{1-x}+\frac{g(-x)}{1+x}+\left(\frac{xf(x)-(1+x)g(x)}{xf(x)+(1+x)g(x)}\right)\left(\frac{g(x)}{1-x}-\frac{g(-x)}{1+x}+\frac{2xf(x)}{1-x^2}\right)\right).\end{align*}
Simplifying this, we find that 
$$(G, F_1, x) \cdot G=\frac{g(x)g(-x)(1-x^2)+x^2f(x)^2}{(1-x^2)^2}.$$
We must now show that 
$$ \frac{(G, F_1, x)\cdot GF_1}{(G, F_1, x)\cdot G}=x.$$
Equivalently, we must show that 
$$(G, F_1, x)\cdot GF_1=x(G, F_1, x)\cdot G.$$ 
Now 
$$(G, F_1,x) \cdot GF_1 =G (GF_1)^e(x^2)+GF_1 (GF_1)^o(x^2).$$
Again, using the hypothesis that $f(-x)=f(x)$, we calculate that 
\begin{scriptsize}
$$(GF_1)^e(x^2)=\frac{x}{2}\left(\left(\frac{g(x)}{x}+\frac{xf(x)}{1-x^2}\right)\left(\frac{xf(x)-(1+x)g(x)}{xf(x)+(1+x)g(x)}\right)-\left(\frac{g(-x)}{1+x}-\frac{xf(x)}{1-x^2}\right)\left(\frac{xf(x)+(1-x)g(-x)}{xf(x)-(1-x)g(-x)}\right)\right).$$
\end{scriptsize} and
\begin{scriptsize}
$$(GF_1)^o(x^2)=\frac{1}{2x}\left(x\left(\frac{g(x)}{1-x}+\frac{xf(x)}{1-x^2}\right)\frac{xf(x)-(1+x)g(x)}{xf(x)+(1+x)g(x)}+x\left(\frac{g(-x)}{1+x}-\frac{xf(x)}{1-x^2}\right)\frac{xf(x)+(1-x)g(-x)}{xf(x)-(1-x)g(-x)}\right).$$ 
\end{scriptsize}
With these expressions, we find that 
$$(G, F_1, x)\cdot GF_1=x(G, F_1, x)\cdot G$$ as required. Thus 
$$M(g,f)^2=\left(\frac{g(x)g(-x)+x^2 f(x)^2}{(1-x^2)^2},x,x\right)=\left(\frac{g(x)g(-x)+x^2 f(x)^2}{(1-x^2)^2},x\right).$$ 
\end{proof}
\begin{corollary} Let $(h(x^2),x)$ be an element of the aerated Appell subgroup of the Riordan group. If $h(x^2)$ can be expressed as
$$h(x^2)=\frac{g(x)g(-x)(1-x^2)+x^2 f(x)^2}{(1-x^2)^2}$$ for $(g,f) \in \mathcal{F}_0 \times \mathcal{F}_0$ with $f(x)$ even, then a square root of the aerated Appell array $(h(x^2),x)$ is given by the Sprugnoli array $M(g,f)$.
\end{corollary}
\begin{example} Let $h(x^2)=\frac{1}{1-x^2}$. This means that we are looking for a square root of the aerated Appell array $\left(\frac{1}{1-x^2},x\right)$. Then assuming that $f(x)=1$, we wish to find $g(x)$ such that
$$\frac{g(x)g(-x)(1-x^2)+x^2}{(1-x^2)^2}=\frac{1}{1-x^2}.$$
Solving for $g(x)g(-x)$, we find that 
$$g(x)g(-x)=\frac{1-2x^2}{1-x^2}=\frac{(1-\sqrt{2}x)(1+\sqrt{2}x)}{(1-x)(1+x)}.$$ 
This gives us $4$ possibilities for $g(x)$. We take, for instance, $g(x)=\frac{1-\sqrt{2}x}{1-x}$. 
Taking the general case of $g(x)=\frac{1-rx}{1-x}$ initially, we find that $M(g,f)=M(g,1)$ begins 
\[
\begin{pmatrix}
1 & 0 & 0 & 0 & 0 & 0 & 0 \\
3 - r & -1 & 0 & 0 & 0 & 0 & 0 \\
3 - 2r & r - 1 & 1 & 0 & 0 & 0 & 0 \\
5 - 3r & 2r - 3 & 3 - r & -1 & 0 & 0 & 0 \\
5 - 4r & 3(r - 1) & 3 - 2r & r - 1 & 1 & 0 & 0 \\
7 - 5r & 4r - 5 & 5 - 3r & 2r - 3 & 3 - r & -1 & 0 \\
7 - 6r & 5(r - 1) & 5 - 4r & 3(r - 1) & 3 - 2r & r - 1 & 1
\end{pmatrix}.
\]
The square of this matrix then begins

\[
\begin{pmatrix}
1 & 0 & 0 & 0 & 0 & 0 & 0 \\
0 & 1 & 0 & 0 & 0 & 0 & 0 \\
3 - r^{2} & 0 & 1 & 0 & 0 & 0 & 0 \\
0 & 3 - r^{2} & 0 & 1 & 0 & 0 & 0 \\
5 - 2r^{2} & 0 & 3 - r^{2} & 0 & 1 & 0 & 0 \\
0 & 5 - 2r^{2} & 0 & 3 - r^{2} & 0 & 1 & 0 \\
7 - 3r^{2} & 0 & 5 - 2r^{2} & 0 & 3 - r^{2} & 0 & 1
\end{pmatrix}.
\]
Now letting $r=\sqrt{2}$, we obtain the matrix $M(g,f)^2=M(g,1)^2$ that begins
\[
\begin{pmatrix}
1 & 0 & 0 & 0 & 0 & 0 & 0 \\
0 & 1 & 0 & 0 & 0 & 0 & 0 \\
1 & 0 & 1 & 0 & 0 & 0 & 0 \\
0 & 1 & 0 & 1 & 0 & 0 & 0 \\
1 & 0 & 1 & 0 & 1 & 0 & 0 \\
0 & 1 & 0 & 1 & 0 & 1 & 0 \\
1 & 0 & 1 & 0 & 1 & 0 & 1
\end{pmatrix}.
\]
Thus \[
\begin{pmatrix}
1 & 0 & 0 & 0 & 0 & 0 & 0 \\
3 - \sqrt{2} & -1 & 0 & 0 & 0 & 0 & 0 \\
3 - 2\sqrt{2} & \sqrt{2} - 1 & 1 & 0 & 0 & 0 & 0 \\
5 - 3\sqrt{2} & 2\sqrt{2} - 3 & 3 - \sqrt{2} & -1 & 0 & 0 & 0 \\
5 - 4\sqrt{2} & 3\sqrt{2} - 3 & 3 - 2\sqrt{2} & \sqrt{2} - 1 & 1 & 0 & 0 \\
7 - 5\sqrt{2} & 4\sqrt{2} - 5 & 5 - 3\sqrt{2} & 2\sqrt{2} - 3 & 3 - \sqrt{2} & -1 & 0 \\
7 - 6\sqrt{2} & 5\sqrt{2} - 5 & 5 - 4\sqrt{2} & 3\sqrt{2} - 3 & 3 - 2\sqrt{2} & \sqrt{2} - 1 & 1
\end{pmatrix}^2\]
\[\quad\quad
=\begin{pmatrix}
1 & 0 & 0 & 0 & 0 & 0 & 0 \\
0 & 1 & 0 & 0 & 0 & 0 & 0 \\
1 & 0 & 1 & 0 & 0 & 0 & 0 \\
0 & 1 & 0 & 1 & 0 & 0 & 0 \\
1 & 0 & 1 & 0 & 1 & 0 & 0 \\
0 & 1 & 0 & 1 & 0 & 1 & 0 \\
1 & 0 & 1 & 0 & 1 & 0 & 1
\end{pmatrix}.
\]
We have 
$$\left(\frac{1-(\sqrt{2}-2)x-(\sqrt{2}+1)x^2}{(1-x)^2(1+x)}, -x\frac{1-(\sqrt{2}-1)x-(2\sqrt{2}-1)x^2-\sqrt{2}x^3}{1-(\sqrt{2}-3)x-(2\sqrt{2}-1)x^2-\sqrt{2}x^3},x \right)^2=\left(\frac{1}{1-x^2},x\right).$$
As we are in a group, the square of the inverse of this array will be equal to $(1-x^2,x)$. 
\end{example}
\section{Involutions in the Sprugnoli group}
We have $(1,x,x)=(1,x)$. Clearly, the identity matrix is an element of the aerated Appell group. Thus in order to find involutions in the Sprugnoli groups, we can seek $(g(x),f(x)) \in \mathcal{F}_0 \times \mathcal{F}_0$ with $f(x)$ even, such that 
$$\frac{g(x)g(-x)(1-x^2)+x^2f(x)^2}{(1-x^2)^2}=1.$$
\begin{example} We take $f(x)=1$ and solve the equation 
$$\frac{g(x)g(-x)(1-x^2)+x^2}{(1-x^2)^2}=1.$$
We find that 
$$g(x)g(-x)=(1-x^2)-\frac{x^2}{1-x^2}=\left(\sqrt{1-x^2}-\frac{x}{\sqrt{1-x^2}}\right)\left(\sqrt{1-x^2}+\frac{x}{\sqrt{1-x^2}}\right).$$
Thus one solution is given by 
$$g(x)=\sqrt{1-x^2}-\frac{x}{\sqrt{1-x^2}}.$$
Hence we obtain the involution in the Sprugnoli group given by
$$\left(\frac{x(1-x)-(1-x-x^2)\sqrt{1-x^2}}{(1-x)(1-x^2)}, -x\frac{(1+x)(1-x-x^2)-x\sqrt{1-x^2}}{(1+x)(1-x-x^2)+x\sqrt{1-x^2}},x\right)^2=(1,x,x).$$
\end{example}
\begin{example} We take $f(x)=1-x^2$. Then we wish to solve 
$$\frac{g(x)g(-x)(1-x^2)+x^2(1-x^2)^2}{(1-x^2)^2}=1.$$
We find that $$g(x)g(-x)=(1+x)^2(1-x)^2.$$
Thus we take $g(x)=(1+x)^2$ along with $f(x)=1-x^2$. 
Then $M(g,f)=\left(\frac{1+3x}{1-x}, -x\frac{1+x+2x^2}{1+3x}, x\right)$ satisfies $M(g,f)^2=I$. 
\[
\begin{pmatrix}
1 & 0 & 0 & 0 & 0 & 0 & 0 \\
4 & -1 & 0 & 0 & 0 & 0 & 0 \\
4 & -2 & 1 & 0 & 0 & 0 & 0 \\
4 & -4 & 4 & -1 & 0 & 0 & 0 \\
4 & -4 & 4 & -2 & 1 & 0 & 0 \\
4 & -4 & 4 & -4 & 4 & -1 & 0 \\
4 & -4 & 4 & -4 & 4 & -2 & 1
\end{pmatrix}^2=
\begin{pmatrix}
1 & 0 & 0 & 0 & 0 & 0 & 0 \\
0 & 1 & 0 & 0 & 0 & 0 & 0 \\
0 & 0 & 1 & 0 & 0 & 0 & 0 \\
0 & 0 & 0 & 1 & 0 & 0 & 0 \\
0 & 0 & 0 & 0 & 1 & 0 & 0 \\
0 & 0 & 0 & 0 & 0 & 1 & 0 \\
0 & 0 & 0 & 0 & 0 & 0 & 1
\end{pmatrix}.
\]
More generally, we can show that 
$$\left(\frac{1+rx}{1-x}, \frac{-x(1+x+(r-1)x^2)}{1+rx}, x\right)^2=I.$$
\end{example}
\begin{example} We now let $f(x)=1-4x^2$.  Then we wish to solve
$$\frac{g(x)g(-x)(1-x^2)+x^2(1-4x^2)^2}{(1-x^2)^2}=1.$$
We find 
$$g(x)g(-x)=\frac{1-3x^2+9x^4-16x^6}{1-x^2}=\frac{(1+x-x^2-4x^3)(1-x-x^2+4x^3)}{1-x^2}.$$ 
Taking $g(x)=\frac{1+x-x^2-4x^3}{1+x}$ now gives us the Sprugnoli array 
$$M(g,f)=\left(\frac{1+2x-x^2+8x^3}{1-x^2}, \frac{-x(1-x^2)}{1+2x-x^2+8x^3}, x\right).$$ This gives 
\[
\begin{pmatrix}
1 & 0 & 0 & 0 & 0 & 0 & 0 \\
2 & -1 & 0 & 0 & 0 & 0 & 0 \\
0 & 0 & 1 & 0 & 0 & 0 & 0 \\
-6 & 0 & 2 & -1 & 0 & 0 & 0 \\
0 & 0 & 0 & 0 & 1 & 0 & 0 \\
-6 & 0 & -6 & 0 & 2 & -1 & 0 \\
0 & 0 & 0 & 0 & 0 & 0 & 1
\end{pmatrix}^2=\begin{pmatrix}
1 & 0 & 0 & 0 & 0 & 0 & 0 \\
0 & 1 & 0 & 0 & 0 & 0 & 0 \\
0 & 0 & 1 & 0 & 0 & 0 & 0 \\
0 & 0 & 0 & 1 & 0 & 0 & 0 \\
0 & 0 & 0 & 0 & 1 & 0 & 0 \\
0 & 0 & 0 & 0 & 0 & 1 & 0 \\
0 & 0 & 0 & 0 & 0 & 0 & 1
\end{pmatrix}.
\]
\end{example}
\section{A Riordan group interpretation of the special mapping}
We have introduced a mapping
$$\mathcal{F}_0 \times \mathcal{F}_0 \longrightarrow \mathbb{C}[[x,y]]$$ 
$$(g,f) \mapsto \frac{g(x)}{(1-x)(1+yx)}+\frac{xf(x)}{(1-x^2)(1-yx)}.$$
The presence of $xf(x)$ on the right allows us to view this mapping in an alternative way. The Riordan group satisfies 
$$R = \mathcal{F}_0 \ltimes \mathcal{F}_1,$$ and so we can define the map 
$$R = \mathcal{F}_0 \ltimes \mathcal{F}_1 \longrightarrow \mathbb{C}[[x,y]],$$
$$ (g, xf) \mapsto B(x,y)=\frac{g(x)}{(1-x)(1+yx)}+\frac{xf(x)}{(1-x^2)(1-yx)}.$$
Now since $B(x,y)$ is the generating function of a Sprugnoli array, we have in fact defined a mapping from the Riordan group to the Sprugnoli group. 
$$ R \longrightarrow S$$
$$ (g,xf) \mapsto \left(\frac{g(x)}{1-x}+\frac{xf(x)}{1-x^2}, x\left(\frac{xf(x)-(1+x)g(x)}{xf(x)+(1+x)g(x)}\right), x\right).$$ 
Our main result now says that if the element $(g,xf) \in R$ is such that $xf$ is odd, then the image of $(g,xf)$ in $S$ has its square in the aerated Appell subgroup of $R$.
\section{Conclusions}
We have shown that Sprugnoli arrays can play an important role in determining the root structure of some aerated Appell arrays. This is an interesting interplay between the Sprugnoli group and the Riordan group. We can ask if this can be generalized: might other Riordan arrays have square roots that are Sprugnoli arrays? We have also found a method for generating involutions in the Sprugnoli group. The crucial role is played by the special mapping from $R$ to $S$. One can ask whether this mapping has other special characteristics, or are there other mappings from $R$ to $S$ that could lead to other insights into elements of the Riordan group itself.

\bigskip
\hrule
\bigskip
\noindent 2020 {\it Mathematics Subject Classification}:
Primary 15B36; Secondary 05A15, 11B83, 11C20, 15A15.

\noindent \emph{Keywords:} Riordan array, Appell subgroup, Sprugnoli group, generating function, square root, involution.

\bigskip
\hrule
\bigskip
\noindent (Concerned with sequences
\seqnum{A000045},
\seqnum{A052952},
\seqnum{A074331},
\seqnum{A085478},
\seqnum{A122367} and
\seqnum{A167270}.)

\end{document}